\documentclass[11pt]{article}

\usepackage[usenames,dvipsnames,condensed]{xcolor}
\usepackage{tikz-cd} 
\usepackage{amsthm}
\usepackage{amsfonts}
\usepackage[UKenglish]{babel}
\usepackage[usenames]{xcolor}
\usepackage{graphicx}
\usepackage{soul}
\usepackage{stfloats}
\usepackage{morefloats}
\usepackage{cite}
\usepackage{lscape}
\usepackage{epstopdf}
\usepackage{mathrsfs}

\usepackage{amsmath,amstext,amsopn,amsfonts,eucal,amssymb,multicol}
\usepackage{graphicx,wrapfig,url}

\newcommand{\Sym}{\mathbb{S}}
\newcommand{\col}[4]{ {\rm Col}_{#1}^{#2,#3}{(#4)}}

\newcommand\Z{{\mathbb Z}}

\newtheorem{theorem}{Theorem}[section]
\newtheorem{corollary}[theorem]{Corollary}
\newtheorem{lemma}[theorem]{Lemma}
\newtheorem{proposition}[theorem]{Proposition}
\newtheorem{definition}[theorem]{Definition}

\newtheorem{example}[theorem]{Example}

\newtheorem{remark}[theorem]{Remark}

\newtheorem{conjecture}[theorem]{Conjecture}

\begin{document}

\title{Computation of 
Quandle 2-Cocycle Knot Invariants Without Explicit 2-Cocycles}

\author{W. Edwin Clark, \  Larry A. Dunning, \ Masahico  Saito \\
Department of Mathematics and Statistics\\ University of South Florida
}

\date{\empty}

\maketitle

\begin{abstract}
We explore a knot invariant derived from  colorings of corresponding $1$-tangles with arbitrary connected quandles. When the quandle is an abelian extension of a certain type the invariant is 
equivalent to the quandle $2$-cocycle invariant. 
We  construct many such abelian extensions using generalized Alexander quandles without explicitly finding $2$-cocycles. This permits the construction of many $2$-cocycle invariants without exhibiting explicit $2$-cocycles. We show that for connected generalized Alexander quandles the invariant is equivalent to Eisermann's knot coloring polynomial.
Computations using this technique show that the $2$-cocycle  invariant distinguishes all of the oriented prime knots up to 11 crossings and most oriented prime knots with 12 crosssings including classification by symmetry: 
 mirror images, reversals, and reversed mirrors. 
\\[5mm]
\end{abstract}

\section{Introduction}\label{sec:intro}

Sets with certain self-distributive operations called  {\it quandles}
have been extensively used in knot theory. 
The {\it fundamental quandle} of a knot
was defined  \cite{Joyce,Mat} in a manner similar to the
fundamental group of a knot. The number of homomorphisms from the fundamental
quandle to a fixed finite quandle has an interpretation as colorings
of knot diagrams by quandle elements, and has been widely used as a
knot invariant. Algebraic homology theories for quandles 
were defined \cite{CJKLS,FRS1} and
 investigated. 
 Extensions of quandles by cocycles have been studied \cite{AG,CENS,Eis3}, and
 invariants constructed from quandle cocycles 
 are applied to various properties of knots and knotted surfaces (see \cite{CKS}
and references therein). 

In this paper all knots will be oriented. Let   $r(K)$ be  the knot $K$ with orientation reversed and $m(K)$ be the mirror image of $K$.
Let $\mathscr{G} = \{1, r, m, rm \}$ be the group of four symmetries of the set of isotopy classes of oriented knots generated by $r$ and $m$.
In \cite{CESY} it was shown that  there are 26 quandles that suffice to distinguish modulo the group $\mathscr{G}$ all pairs of prime knots with crossing number at most 12 using the number ${\rm Col}_Q(K)$ of colorings of a knot diagram $K$ by the quandle $Q$.
It is known that ${\rm Col}_Q(K) = {\rm Col}_Q(rm(K))$ for any quandle $Q$, but quandle $2$-cocycle invariants do distinguish $rm(K)$ from $K$ 
for some quandles and for some knots. Thus we pose the following (notation will be specified in Section~\ref{sec:prelim}):

\begin{conjecture} \label{conjecture1} 
The 2-cocycle invariant $\Phi_\phi$ is a complete invariant for oriented knots, that is, 
if $K_1$ and $K_2$ are non-isotopic oriented knots, then 
there exists a finite quandle $Q$ and a $2$-cocycle $\phi$ such that $\Phi_\phi(K_1) \neq \Phi_\phi(K_2)$.
\end{conjecture}

In this paper we examine this conjecture computationally for prime knots with at most 12 crossings using  colorings of $1$-tangles
to compute $2$-cocycle invariants without explicitly finding $2$-cocycles. Since ${\rm Col}_Q(K)$ is equal to the $2$-cocycle invariant for a coefficient group of order 1, it follows from the results of \cite{CESY} that the conjecture holds when $K_1$ and $K_2$ are prime knots with crossing number at most 12 that lie in distinct $\mathscr{G}$-orbits. Thus we need only consider the cases where  $K_1$ and $K_2$ lie in the same $\mathscr{G}$-orbit. For this purpose, effective ways of constructing abelian extensions by generalized Alexander quandles 
are studied, and a set of 60 quandles is presented that 
distinguishes all distinct $K_1$ and $K_2$ in the same $\mathscr{G}$-orbit
for all except 13 knots with up to 12 crossings. 

In Section~\ref{sec:prelim}, preliminary material is reviewed.
The invariant using colorings of $1$-tangles is defined in Section~\ref{sec:def}, and its relations to other invariants are 
given in Section~\ref{sec:rel}. 
Non-faithful quandles play a key role in the invariant, and they are discussed in Section~\ref{sec:nonfaith}. 
In Section~\ref{sec:compute}, computational outcomes are discussed.
Theorems needed for the construction of quandles used in this paper are proved in appendices.

\section{Preliminaries} \label{sec:prelim}

In this section we briefly review some definitions and examples. 
More details can be found, for example, in \cite{CKS}. 

A {\it quandle} $Q$ is a set with a binary operation $(a, b) \mapsto a * b$
satisfying the following conditions.
\begin{eqnarray}
\mbox{\rm (Idempotency) } & &  \mbox{\rm  For any $a \in Q$,
$a* a =a$.} \label{axiom1} \\
\mbox{\rm (Right invertibility)}& & \mbox{\rm For any $b,c \in Q$, there is a unique $a \in Q$ such that 
$ a*b=c$.} \label{axiom2} \\
\mbox{\rm (Right self-distributivity)} & & 
\mbox{\rm For any $a,b,c \in Q$, we have
$ (a*b)*c=(a*c)*(b*c). $} \label{axiom3} 
\end{eqnarray}
   
  Let $Q$ be a quandle.
  The {\it right translation}  ${R}_a:Q\rightarrow  Q$, by $a \in Q$, is defined
by ${ R}_a(x) = x*a$ for $x \in Q$. 
Then ${ R}_a$ is an automorphism of $Q$ by Axioms (2) and (3). 
The subgroup of ${\rm Sym}(Q)$ generated by the permutations ${ R}_a$, $a \in Q$, is 
called the {\it {\rm inner} automorphism group} of $Q$,  and is 
denoted by ${\rm Inn}(Q)$. 
The map ${\rm inn}: Q \rightarrow {\rm inn}(Q) \subset {\rm Inn}(Q)$
defined by ${\rm inn}(x)=R_x$ is called the {\it inner representation}.

A quandle is {\it connected} if ${\rm Inn}(Q)$ acts transitively on $Q$.
A quandle is {\it faithful} if the mapping ${\rm inn}: Q \rightarrow  {\rm Inn}(Q)$ is an injection.

As in Joyce \cite{Joyce},
given a group $G$ and   $f \in {\rm Aut}(G)$,
one can define a quandle operation on $G$ by 
$x*y=f(xy^{-1}) y ,$ $ x,y \in G$.  
We call such a quandle a {\it generalized Alexander quandle} and denote it by  ${\rm GAlex}(G,f)$.
If $G$ is abelian, such a quandle is  known as an {\it Alexander quandle}.

If  $H$ is a subgroup of ${\rm Fix}(G,f)$, then ${\mathscr H}(G,H,f) $ is  the quandle  with underlying set $\{Hg: g \in G\}$ with product
given by $Ha*Hb = Hf(ab^{-1})b$. (We use ${\mathscr H}$ for a {\it homogeneous} quandle, instead of 
${\mathscr Q}_{\rm Hom}$ used in \cite{HSV}. This construction has been used by various authors under different names.) Clearly we have ${\rm GAlex}(G, f)={\mathscr H}(G,\{1\},f)$.

 A {\it quandle homomorphism} between two quandles $X, Y$ is
 a map $f: X \rightarrow Y$ such that $f(x*_X y)=f(x) *_Y f(y) $, where
 $*_X$ and $*_Y$ 
 denote 
 the quandle operations of $X$ and $Y$, respectively.
 A {\it quandle isomorphism} is a bijective quandle homomorphism, and 
 two quandles are {\it isomorphic} if there is a quandle isomorphism 
 between them.
 A quandle epimorphism $f: X \rightarrow Y$ is a {\it covering}~\cite{Eis3}
 if $f(x)=f(y)$ implies $a*x=a*y$ for all $a, x, y \in X$. 
 An inner representation is a covering.
We say that two homomorphisms $f: X \rightarrow Y$ and $g:Z \rightarrow W$ are {\it equivalent }
if there are isomorphisms $i: X \rightarrow Z$ and $j:Y \rightarrow W$ such that  the following diagram is commutative.

\[ \begin{tikzcd}
X \arrow{r}{f} \arrow[swap]{d}{i} & Y \arrow{d}{j} \\
Z \arrow{r}{g}& W
\end{tikzcd}
\]

Let $D$ be a diagram of a knot $K$, and ${\cal A}(D)$ be the set of arcs of $D$.
A {\it  coloring}  of a knot diagram $D$ by a quandle $Q$
is a map $C: {\cal A}(D) \rightarrow Q$  satisfying the condition depicted in Figure~\ref{coloredXing}
at every 
 positive (left) and negative (right) crossing $\tau$,
respectively.  The pair $(x_\tau, y_\tau)$ of colors assigned to a pair of nearby arcs of a crossing $\tau$
is called the {\it source} colors, and the third arc is required to receive the color $x_\tau * y_\tau$.

\begin{figure}[htb]
    \begin{center}
   \includegraphics[width=3in]{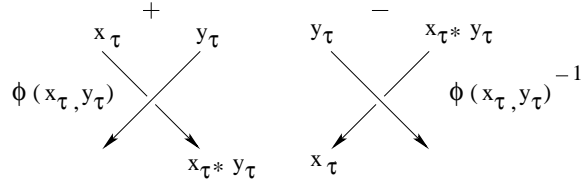}\\
    \caption{Colored crossings and cocycle weights }\label{coloredXing}
    \end{center}
\end{figure}

A  $1$-{\it tangle} (also called a long knot) is a properly embedded arc in a $3$-ball, and the equivalence of
$1$-tangles is defined by ambient  isotopies of the $3$-ball fixing the
boundary (cf.~\cite{Conway}).  A diagram of a $1$-tangle  is defined in a
manner similar to a knot diagram, from a regular projection to a disk by
specifying  crossing information. 
 As in \cite{CS} we assume that the 1-tangles are oriented from top to bottom.
  A knot
diagram is obtained from a $1$-tangle diagram by closing the end points by a
trivial arc outside of a disk. This procedure is called the {\it closure} of a
$1$-tangle.  If a $1$-tangle is oriented, then the closure inherits the
orientation.
Two  diagrams of the same $1$-tangle are related by Reidemeister moves.
As indicated, for example in \cite{Eis1}, 
there is a bijection from isotopy classes of knots to those of the $1$-tangles, corresponding to the closure.
Thus  an invariant of a 1-tangle $T$ corresponding to a knot $K$ is an invariant of $K$.
 
 For simplicity we often identify a $1$-tangle $T$  with a diagram of $T$ and similarly for knots. 
A quandle coloring of an oriented $1$-tangle diagram is defined in a  manner
similar to  those for knots.  We do not require that the end points receive the
same color for a quandle coloring of $1$-tangle diagrams.
 As in \cite{CSV} we say that a quandle $Q$ is {\it end monochromatic} for a tangle diagram $T$ if 
any coloring of $T$ by $Q$ assigns the same color on the two end arcs.

\begin{figure}[htb]
    \begin{center}
   \includegraphics[width=1.1in]{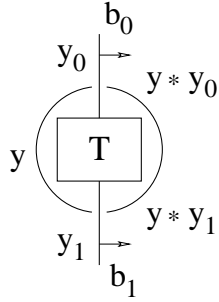}\\
    \caption{ Colorings of a  tangle }\label{endmono}
    \end{center}
\end{figure}

We will need the following two lemmas.
 
\begin{lemma}[Eisermann~\cite{Eis3}, Theorem 30] \label{lem:cover}
Let $f: Y \rightarrow X$ be a covering, and 
$C_X: {\cal A} (T) \rightarrow X$ be a coloring of a $1$-tangle $T$ by $X$. 
Let $b_0, b_1$ be the top and bottom arcs as depicted in Figure~\ref{endmono}. 
Then  for any $y \in Y $ such that $f(y)=C_X(b_0)$,  
there exists a unique coloring $C_Y: {\cal A} (T) \rightarrow Y$ such that 
$f  C_Y=C_X$ and $C_Y(b_0)=y$.
\end{lemma}

The proof of the following crucial lemma established  in \cite{Nos,Jozef,Jozef2004} for faithful quandles  extends easily to non-faithful quandles. The idea of proof is seen in Figure~\ref{endmono}.

\begin{lemma}\label{lem:end}
Let  $C:  {\cal A} (T) \rightarrow Y$ be a coloring of a classical $1$-tangle diagram $T$ by a quandle $Y$. 
For the top and bottom arcs $b_0$ and $b_1$ of $T$, respectively, 
let $y_0=C(b_0)$ and $y_1=C(b_1)$.
Then  ${\rm inn}(y_0)=R_{y_0}=R_{y_1}={\rm inn}(y_1)$. 
\end{lemma}

\section{Invariants from coloring $1$-tangles}\label{sec:def} 

In this section we define a knot invariant using quandle colorings of corresponding $1$-tangles.
  Let $T$ be a 1-tangle whose closure is the knot $K$. 
 Let ${\cal A}(T) $ denote the set of arcs of the $1$-tangle diagram $T$.  We 
 denote the top arc of $T$ by $b_0$ and the bottom arc by $b_1$ as in Figure~\ref{endmono}.
 For arbitrary fixed $e \in Q$ denote the set of colorings  $C: {\cal A}(T) \rightarrow Q$ by a quandle $Q$ 
 such that $C(b_0)=e$ by ${\rm Col}^e_Q(T)$. 

For a connected quandle $Q$ and elements $a, b \in Q$, let $\col{Q}{a}{b}{T}$ be the number of colorings $C$ of $T$ such that  $C(b_0) = a$ and $C(b_1) = b$.  Since the 1-tangle $T$ is uniquely determined up to isotopy by $K$, 
$\col{Q}{a}{b}{K}  := \col{Q}{a}{b}{T}$  is an invariant of $K$. 
It is clear that if $f$ is an automorphism of $Q$ then the mapping $C \mapsto fC$  is a one-to-one correspondence between the  colorings $C$ of $T$ with $C(b_0) = a$, $C(b_1) = b$ and the colorings $C'$ of $T$ with $C'(b_0) = f(a)$, $C'(b_1) = f(b)$. Hence we have $\col{Q}{a}{b}{T} = \col{Q}{f(a)}{f(b)}{T}$.  
Consider the action of the group ${\rm Inn}(Q)$ on the set of pairs $Q \times Q$ in the obvious way, namely, for $g \in {\rm Inn}(Q)$ and $(a,b) \in Q \times Q$ set $g(a,b) = (g(a),g(b))$. For convenience write $(a,b) \sim (c,d)$ if $(a,b)$ and $(c,d)$ are in the same orbit of this action.  Thus we have that 
if $(a,b) \sim (c,d)$ then $\col{Q}{a}{b}{T} = \col{Q}{c}{d}{T}$. 
Hence 
it suffices to consider only the invariants  $\col{X}{a}{b}{T}$ where $(a,b)$ runs through a set of representatives of the orbits of pairs under the action of ${\rm Inn}(X)$.  
We assume $Q$ is connected so that ${\rm Inn}(Q)$ acts transitively on $Q$. Thus we may fix $a=e$ for an arbitrarily chosen $e \in Q$, and  then every orbit has a representative of the form $(e,b)$ for some $b \in Q$.   
Furthermore, by Lemma~\ref{lem:end}, for $C \in {\rm Col}_Q^e(T)$,  $b = C(b_1)$  satisfies $R_b = R_e$. That is, $b$ lies the the fiber $F_e = {\rm inn}^{-1}(R_e)$.  Thus we define the following invariant.

\begin{definition}
{\rm 
 For arbitrary fixed $e \in Q$ denote the set of colorings  $C: {\cal A}(T) \rightarrow Q$ by quandle $Q$ 
 such that $C(b_0)=e$ by ${\rm Col}_Q^e(T)$. Define 
$$\Psi_Q^e(K) = \sum_{C \in  {\rm Col}_Q^e(T) } C(b_1). $$
We think of $\Psi_Q^e(K)$ as lying in the free $\Z$-module $\Z[F_e]$  with basis $F_e$. 
}
\end{definition}

\begin{remark}
{\rm
Our motivation for defining $\Psi^e_Q(K)$ came from the proof of Proposition~\ref{connected_sum} below. 
Later we discovered the relation to Eisermann's knot coloring polynomial discussed in Section~\ref{sec:colorpoly}.
}
\end{remark}

Computation of the invariant $\Psi^e_Q(K)$ is performed as follows.
For a quandle $Q$ of order $n$ we may, by relabeling if necessary, assume that the elements of $Q$ are the integers $1,2,\ldots,n$.  
By the  matrix $M$ of $Q$ we mean the $n \times n$ matrix $M$ such that $M_{i,j} = i*j$ where $*$ denotes the quandle product.    Note that $R_i = R_j$ 
means that  the $i$\/th and $j$\/th columns of $M$ are equal. By relabeling  the quandle elements we can assume that 
$\{ j : R_1 = R_j \} = \{1,2,\dots,s \}$. 
Also again by relabeling if necessary we may assume that for $ i,j \in \{1,2,\dots,s \}$  we have that $(1,i) \nsim (1,j)$. With this convention we may think of the knot  invariant $\Psi^1_Q(K)$ as the  vector in $\Z^s$ given by
$$\Psi^1_Q(K)  = (\col{Q}{1}{1}{T},\col{Q}{1}{2}{T},\dots,\col{Q}{1}{s}{T})$$ 
where $T$ is a tangle corresponding to $K$.  Note that if $s = 1$ then the quandle $Q$ is faithful and $\Psi^1_Q(K) =\col{Q}{1}{1}{T}$ is just the number of colorings of the knot $K$ by $Q$ divided by $|Q|$.

We formulate the relation between $\Psi^e_Q(K)$ and $\Psi^e_Q( rm (K) )$ as follows. 
First we recall such a formula for the cocycle invariant $\Phi_\phi(K)$ (see Section~\ref{sec:cocyinv} below for the definition).
 For an abelian group $A$ and an element $g= \sum_h g_h h$, $g_h \in \Z$, $h \in A$, 
in the group ring $\Z [A]$, the element
	$\overline{g}=\sum_h g_h h^{-1} \in \Z [A]$ is called the {\it conjugate} of
	$g$.
In \cite{CENS}  it is established that 
  $$\Phi_{\phi} (rm(K))= \overline{ \Phi_{\phi} (K) }.$$ 
In other words, after computing $\Phi_{\phi} (K)$ one essentially gets  $\Phi_{\phi} (rm(K))$  for free.

We now  show that the components of  $\Psi^1_Q(rm(K))$ are a permutation $p$ of the components of  $\Psi^1_Q(K)$ depending on the quandle  $Q$ only. 
Thus, to determine whether or not the invariant can distinguish $K$ from $rm(K)$ it is only necessary to compute $\Psi^1_Q(K)$ and the permutation $p$.

\begin{lemma}
For any knot $K$, there exists a permulation $p$ of order $2$ on $\{1, \ldots, s\}$ such that 
  $$\Psi^1_Q(rm(K))  = (\col{Q}{1}{p(1)}{T},\col{Q}{1}{p(2)}{T},\dots,\col{Q}{1}{p(s)}{T}).$$ 
\end{lemma}

\begin{proof}
Let  $T$ be a 1-tangle and $rm(T)$ be its reverse mirror image. If $b_0,b_1$ are, respectively,
the top and bottom arcs of $T$  then  $b_1,b_0$ are, respectively, the top and bottom arcs of $rm(T)$.
Then the colorings $C$ by $Q$ of $T$ with $C(b_0) = 1$ and $C(b_1) = j$ are in one-to-one
correspondence with the colorings $C'$ of $rm(T)$ satisfying  $C'(b_0) = j$ and $C'(b_1) = 1$.
That is, $\col{Q}{1}{j}{T}=\col{Q}{j}{1}{rm(T)}$ for $j \in \{1,2,\dots,s\}$.
 We note that by connectedness there is for each $j \in Q$ an automorphism $f_j$ such that 
$f_j(j) = 1$,  so then we have $\col{Q}{j}{1}{rm(T)} = \col{Q}{1}{f_j(1)}{rm(T)}$.  It follows that 
 $$\col{Q}{1}{f_j(1)}{rm(T)}= \col{Q}{1}{j}{T}$$
for all  $j \in \{1,2,\dots,s\}$.  Note that clearly $R_1 = R_j$ implies that $R_{f_j(1)} = R_{f_j(j)} = R_1$ and hence $f_j(1) \in \{1,2, \dots,s \}$. Since the pairs $(1,j)$, $ j =1,2,\dots, s$ are in distinct
orbits by assumption  we must have 
$\{f_1(1),f_2(1),\dots,f_s(1)\} = \{1,2,\dots,s \}$.   That is,
there is a permutation $p$ of $\{1,2,\dots,s \}$ such that
$p(i) = f_i(1)$ and we have 
     $$\col{Q}{1}{p(j)}{rm(T)}= \col{Q}{1}{j}{T}$$
or
     $$\col{Q}{1}{j}{rm(T)}= \col{Q}{1}{p^{-1}(j)}{T} . $$
Hence we obtain the formula by replacing $p^{-1}$ by $p$.  It is clear by the construction that $p$ is an involution.
\end{proof}

 \begin{remark}
 {\rm 
Note that since $p(1) = f_1(1) = 1$, if $s = 2$ for a quandle $Q$ then $p = {\rm identity}$ so  the invariant $\Psi^e_Q$ will not be able to distinguish $K$ from $rm(K)$. In particular this is true for any non-faithful quandle $Q$ such that
 $|Q|/|{\rm inn}(Q)| = 2$.  
 Quandles  $Q$ with $|Q|/|{\rm inn}(Q)| > 2$ will not be able to distinguish $K$ from $rm(K)$ whenever $p = {\rm identity}$. 
This is the case for the Rig quandles (see Section~\ref{sec:nonfaith} below)  $Q(30,4)$, $Q(36,57)$, $Q(36,58)$ and $Q(45,20)$ each of which has $|Q|/|{\rm inn}(Q)| = 3$.  This helps to eliminate efficiently many quandles $Q$ for which $\Psi^e_Q$ cannot distinguish $K$ from $rm(K)$.
}
\end{remark}

Next we compare this invariant with the colorings of connected sums studied in  \cite{CSV}. 
There an invariant $K \mapsto {\rm Col}_Q(K \#P)$ was defined for each knot $P$ and quandle $Q$, where ${\rm Col}_Q(K \#P)$ is the number of colorings of the connected sum $K \#P$ by the quandle $Q$.

\begin{proposition} \label{connected_sum} 
\begin{sloppypar}
If for knots $K_1$ and $K_2$ there is a quandle $Q$ and a knot $P$ such that ${\rm Col}_Q(K_1 \# P) \neq {\rm Col}_Q(K_2 \# P)$, then $ \Psi^e_Q(K_1) \ne \Psi^e_Q(K_2).$
\end{sloppypar}
\end{proposition}

\begin{proof} Let $T_i$ be a 1-tangle whose closure is $K_i$,  $i = 1,2$ and let $T_P$ be a 1-tangle whose closure is $P$. Let $b_0^i, b_1^i$ be, respectively, the top and bottom arcs of $T_i$  and let $c_0, c_1$ be, respectively, the top and bottom arcs of $T_P$. Then the 1-tangle $T(K_i \#P)$ whose closure is $K_i \# P$ can be obtained by joining the bottom arc $b_1^i$ of $T_i$ to the  top arc $c_0$ of $T_P$. Then  it is clear that for $i = 1,2$ we have
$${\rm Col}_Q(K_i \# P) = |Q|{\rm Col}_Q^{e,e}(T(K_i \#P))= |Q|\sum_{a \in Q} {\rm Col}_Q^{e,a}(T_i) {\rm Col}_Q^{a,e}(T_P) .$$
It follows that if $\Psi^e_Q(K_1) = \Psi^e_Q(K_2)$ then ${\rm Col}_Q(K_1 \# P) = {\rm Col}_Q(K_2 \# P)$, as we wanted to show.
\end{proof}

\begin{remark} 
{\rm
Let  $\mathscr{Q}(K)$  denote the fundamental quandle (or knot quandle) (\cite{Joyce},\cite{Mat} ) of the knot $K$. It is known that the quandles $\mathscr{Q}(K_1)$ and $\mathscr{Q}(K_2)$ are isomorphic if and only if $K_1 = K_2$ or $K_1 = rm(K_2).$ Since colorings of a knot $K$ by a quandle $Q$ are equivalent to homomorphisms from $\mathscr{Q}(K)$ to $Q$, clearly one cannot expect to distinguish $K$ from $rm(K)$ using the number of colorings ${\rm Col}_Q(K)$ of $K$  by $Q$ alone. On the other hand, it was shown in \cite{CSV} that given
two oriented knots $K_1$ and $K_2$ such that $K_1 \neq K_2$, 
there is a knot $P$ such that 
 $P \# K_2 \neq rm(P \# K_1)$ and hence we have that the fundamental quandles $\mathscr{Q}(P \# K_2)$ and $\mathscr{Q}(P \# K_1)$ are not isomorphic. 
This together with Proposition~\ref{connected_sum} leads to some hope that $\Psi^e_Q$ may be a complete invariant. 
}
\end{remark}

\begin{conjecture} \label{conjecture2} The invariant $\Psi^e_Q$ is a complete invariant for oriented knots, that is,
for two oriented knots $K_1$ and $K_2$, if   $K_1 \neq K_2$ then there exists a finite connected quandle $Q$ such that 
 $\Psi^e_Q(K_1) \neq \Psi^e_Q(K_2)$.
\end{conjecture}

If Conjecture~\ref{conjecture1} holds, then this conjecture holds.
This conjecture corresponds to Eisermann's remark on the possibility of his knot coloring polynomial being a classifying invariant (Question~1.5 in \cite{Eis2}), based on the characterization of knots by the knot groups with their peripheral systems and the residually finiteness of knot groups.

\section{Relations to other invariants}\label{sec:rel} 

As mentioned earlier, the invariant $\Psi^e_Q(K)$ is defined for generalizing and computing 
the quandle cocycle invariant. In this section we clarify the relationship to the cocycle invariant, and to the knot coloring 
polynomial defined by Eisermann~\cite{Eis2}.
Similar results can be found in \cite{Eis2}. We summarize the relations in terms of our notation and approach below.

\subsection{Relation to the cocycle invariant}\label{sec:cocyinv}

Let $\Lambda$  be a not necessarily abelian group with identity denoted by $1$.
 A function $\phi: X \times X \rightarrow \Lambda$ is called
    a \emph{quandle $2$-cocycle} if 
$$\phi(a,b)\phi(a*b,c) = \phi(a,c)\phi(a*c,b*c)$$
    for all $a,b,c \in Q$ and
    $\phi(a,a)=1$ for all $a \in X$.  For a quandle $2$-cocycle $\phi$, $\Lambda \times X$ is a quandle under the product
    \[
    (\lambda,a) * (\mu, b)=(\lambda \phi(a,b),a*b)
    \]
    for $a,b \in X$, $\lambda, \mu \in \Lambda.$ We denote this quandle  by
    $\Lambda \times_\phi X$. As in \cite{CENS} when $\Lambda$ is abelian it is called an \emph{abelian  extension} of $X$ by $\Lambda$ with respect to $\phi$. 

Simple
examples where $\Lambda$ is non-abelian are given in Example \ref{example}. We found many other examples among the generalized Alexander quandles that we constructed using GAP.

Let $\pi: \Lambda \times_\phi X \rightarrow X$ be the projection onto $X$. This is clearly a homomorphism and the
fibers $\pi^{-1}(x)$ are the sets $\Lambda \times \{x\}$. 
There is an action of $\Lambda$ on $\Lambda \times_\phi X$ given by $\lambda(\mu,x) = (\lambda \mu, x)$
(Definition 3.2 \cite{Eis1}).  It is clear that this action is free and transitive on each fiber  $\Lambda \times \{x\}$. So, for example, if $e = (1,x)$ then each element of $F_e = \Lambda \times \{x\}$ can be written uniquely in the form $\lambda e$ for $\lambda \in \Lambda$.

We recall the  definition from  \cite{CJKLS} of the 2-cocycle invariant $\Phi_\phi(K)$ of a knot $K$.
Let $X$ be a quandle, and $\phi$ be a $2$-cocycle with finite abelian coefficient group $\Lambda$.
 
The $2$-cocycle invariant (or state sum invariant)  is the element of the group ring $\Z [\Lambda]$ 
defined by 
$$\Phi_{\phi} (K) = \sum_{C} \prod_{\tau} \phi(x_\tau, y_\tau)^{\epsilon(\tau)},$$
 where
the product ranges over all crossings $\tau$, the sum ranges over all colorings of a 
given knot diagram,
$(x_\tau, y_\tau)$ are source colors at the crossing $\tau$, and $\epsilon(\tau)$ 
is the sign of $\tau$ as specified in Figure~\ref{coloredXing}.
For a given coloring $C$, we write $B_\phi (K, C) = \prod_{\tau} \phi(x_\tau, y_\tau)^{\epsilon(\tau)} \in \Lambda$. 
Thus we may write
 $$\Phi_\phi(K) = \sum_{\lambda \in \Lambda} n_\lambda \lambda$$
 where $n_\lambda$ is the number of colorings of $K$ by $X$ such that $B_\phi (K, C) =\lambda.$ 

\begin{lemma}\label{CocycleInvariant1} (Cf. Eisermann \cite{Eis2}) Let $Q =\Lambda \times_\phi X$ be a connected  abelian extension of a quandle $X$ with respect to a 2-cocycle $\phi.$  Let $T$ be a 1-tangle whose closure is $K$ with top arc $b_0$ and bottom arc $b_1$. For $x \in X$ let  ${\rm Col}_X^{x, x}(T)$ denote the set $C$ of colorings of $T$ by $X$ such that $C(b_0) = x$ and 
$C(b_1) =  x$. If $$\Phi_{\phi}^x (K) = \sum_{C \in {\rm Col}_X^{x, x}(T)} B_\phi(K,C),$$ then
$$\Phi_\phi(K) = \sum_{y \in X} \Phi_\phi^y(K) = |X| \Phi_\phi^x (K) .$$
\end{lemma}
\begin{proof} It follows from the proof of Theorem 4.1 of \cite{CENS} that every coloring $C$ of $T$ by $X$ such  that $C(b_0) = x$ and $C(b_1) = x$ can be uniquely lifted to a coloring $C'$ of $T$  by $Q$ with $C'(b_0) = (1,x)$ and $C'(b_1) = (B_{\phi}(K,C), x)$ (see also Lemma~\ref{lem:cover}). 
Let ${\rm Col}_Q^{(1,x), (\lambda,x)}(T)$ denote the set of all colorings $C'$ of $T$ by $Q$ such that $C'(b_0) = (1,x)$ and $C'(b_1) = (\lambda,x)$ for some $\lambda \in \Lambda$.  This implies that
$$\Phi_\phi^x = \sum_{C \in {\rm Col}_X^{x, x}(T)} B_\phi(K,C) = \sum_{C' \in {\rm Col}_Q^{(1,x), (\lambda,x)}(T)} \pi_1(C'(b_1)), $$
where $\pi_1: (\lambda,x) \mapsto \lambda$.
There is an action of $\Lambda$ on $Q$ given by $\lambda (\mu, a) = (\lambda \mu,a)$. It is easy to see that this action commutes with all right translations $R_{(\tau,a)}$, $(\tau,a) \in Q$,  and hence commutes with all $\varphi \in {\rm Inn}(Q)$ (Cf. Lemma~\ref{coverings} (i)). Let $y \in X$, then since $Q$ is connected, there is a $\varphi \in {\rm  Inn}(Q)$ such that $\varphi(1,x) = (1,y)$. Now
 $$\varphi(\lambda,x) = \varphi( \lambda(1,x)) = \lambda \varphi (1,x) = \lambda (1,y) = (\lambda, y).$$
Clearly the mapping $C' \mapsto \varphi C'$ is a bijection from $ {\rm Col}_Q^{(1,x), (\lambda,x)}(T)$ to $ {\rm Col}_Q^{(1,y), (\lambda,y)}(T)$ and $\pi_1(C'(b_1)) = \pi_1(\varphi C'(b_1))$.
It follows that $\Phi_\phi^x = \Phi_\phi^y$. Since the set of all colorings of $K$ by $X$ is the disjoint union of the 
$ {\rm Col}_X^{x, x}(T)$  for $x \in X$, the result follows. 
\end{proof}

\begin{lemma}\label{CocycleInvariant2} Let $Q =\Lambda \times_\phi X$ be a connected  abelian extension of a quandle $X$ with respect to a 2-cocycle $\phi.$ Let $\pi$ be the projection of $Q$ onto $X$.   Let $T$ be a 1-tangle whose closure is $K$ with top arc $b_0$ and bottom arc $b_1$. Let $\Phi_\phi(K) = \sum_{\lambda \in \Lambda} n_\lambda \lambda$ and for some $x \in X$, let  $\mathscr{C}_0 = \mathscr{C}(Q,x)$ denote the set $C$ of colorings of $T$ such that $C(b_0) = e = (1,x)$ and 
$\pi C(b_1) =  x$ and let $\mathscr{C}_1 = {\rm Col}^e_Q(T) - \mathscr{C}(Q,x)$ then 
$$\Psi^e_Q(K) = \frac{1}{|X|}\sum_{C \in \mathscr{C}_0} n_\lambda (\lambda, x) + \sum_{C \in \mathscr{C}_1} C(b_1).$$
\end{lemma}
\begin{proof} Since $ {\rm Col}^e_Q(T)$ is the disjoint union of $\mathscr{C}_0$ and $\mathscr{C}_1$, we have:
$$\Psi^e_Q(K) = \sum_{C \in \mathscr{C}_0} C(b_1) + \sum_{C \in \mathscr{C}_1} C(b_1).$$
By the proof of Lemma~\ref{CocycleInvariant1} we have
$$\sum_{C \in \mathscr{C}_0} C(b_1) =\sum_{\lambda \in \Lambda} \frac{ n_\lambda}{|X|} (\lambda, x) .$$ 
\end{proof}

\begin{corollary}
If $\Phi_\phi(K_1) \neq \Phi_\phi(K_2)$  for knots $K_1$ and $K_2$, then $\Psi^e_Q(K_1) \neq \Psi^e_Q(K_2)$.  
\end{corollary}

\begin{theorem}\label{AbelExt} Let $Q =\Lambda \times_\phi X$ be a connected  abelian extension of quandle $X$ corresponding to the 2-cocycle $\phi.$  
If  the projection $\pi: Q \rightarrow X$, $(\lambda,x) \mapsto x$, is equivalent to  ${\rm inn}: Q \rightarrow {\rm inn}(Q)$, then for $e = (1,x_0)$  we have for all knots $K$:
 $$\Psi_Q^e(K)= \frac{1}{|X|} \Phi_\phi(K)$$
if one identifies the fiber $F_e$ with $\Lambda$  via $\lambda \mapsto (\lambda,x_0)$.  In particular this formula holds if $X$  is faithful.

\end{theorem}

\begin{proof} From  Lemma~\ref{CocycleInvariant2} it suffices to show that  ${\rm Col}^e_Q(T) = \mathscr{C}_0$ when $\pi$ is equivalent to  ${\rm inn}$.  Let $T$ be a 1-tangle with top arc $b_0$ and bottom arc $b_1$ whose closure is $K$. Now $e = (1,x_0)$ so it suffices to show that if $C$ is a coloring of $T$ such that $C(b_0) = (1,x_0)$ then $C(b_1) = (\lambda,x_0)$ for some $\lambda \in \Lambda$.  We know from \cite{CS} that in any case if $C(b_1) = (\lambda,x)$, $x \in X$, then
$R_{(1,x_0)} = R_{(\lambda,x)}.$
If  $\pi$ is equivalent to ${\rm inn}$ there is an isomorphism $\varphi: {\rm inn}(Q) \rightarrow X$ such that 
$\varphi( {\rm inn}(\lambda,x)) = \pi(\lambda,x) = x$.  And since $\varphi$ must be well-defined we must have that if $R_{(\lambda,x)} = R_{(1,x_0)}$ then $x = x_0$
by Lemma~\ref{lem:end}. 
This shows that ${\rm Col}^e_Q(T)  = \mathscr{C}_0,$ as desired.
 We also have that the fiber ${\rm inn}^{-1}(R_e) = F_e =  \{ (\lambda,x_0) : \lambda \in \Lambda \}$, which justifies identifying $\lambda$ with $(\lambda,x_0)$.

Suppose that $X$ is faithful. Then for $a,b \in X$ we have $R_a = R_b \iff a = b$. By \cite{CS} Lemma 4.1 it suffices to prove that the mapping $R_{(\lambda,a)} \mapsto a$ is a bijection from ${\rm inn}(Q)$ to $X$. To see that this is a bijection observe that 
$R_{(\lambda,a)} = R_{(\mu,b)}$ $\iff$
$ (\gamma,c)*(\lambda,a) = (\gamma,c)*(\mu,b) \  \mbox{for all } (\gamma,c)$ $\iff$
$ (\gamma \phi(c,a),c*a) = (\gamma \phi(c,b),c*b)\  \mbox{for all } (\gamma,c)$ 
$\iff a = b$.
\end{proof}

\subsection{Relation to Eisermann's knot coloring polynomial} \label{sec:colorpoly} 

We recall the definition from \cite{Eis2} of Eisermann's {\it knot coloring polynomial} $P_G^x(K)$ of a knot $K$.  Let $\pi_K$ be the knot group, that is, the fundamental group of the complement of knot $K$. Let $m_K$ be the meridian of $K$ and
let $l_K$ be the longitude of $K$.  For a pointed finite group $(G,x)$, 
$$
P_G^x(K) = \sum_{\rho} \rho(l_K),
$$
where the sum is taken over all homomorphisms $\rho: \pi_K \rightarrow G$ with $\rho(m_K) = x$.  It turns out (see \cite{Eis2}) that the values $\rho(l_K)$ lie in the {\it longitude group} $\Lambda = C(x) \cap G'$. Thus $P^x_G(K)$ lies in the group ring $\Z[\Lambda]$.  Recall that $\Psi^e_Q(K)$ lies in the free $\Z$-module $\Z[F_e]$  with basis $F_e = {\rm inn}^{-1}(R_e)$.

In this section we show that if $Q$ is a connected generalized Alexander quandle, $e \in Q$, $G ={\rm  Inn}(Q)$ and $x = R_e$
then the invariants $\Psi^e_Q(K)$ and $P_G^x(K)$ are equivalent.  Here for $g \in {\rm Inn}(Q)$ and $a \in Q$
we write $ag$ for the image of $a$ under the automorphism $g$.

\begin{lemma} \label{lem:fiber}
Let $Q$ be a connected generalized Alexander quandle, $e \in Q$, $G ={\rm  Inn}(Q)$ and $x = R_e$.
Then the mapping $\pi_e: C(x) \cap G' \rightarrow {\rm inn}^{-1} (R_e)$ given by 
$\pi_e(g) = eg$, is a bijection.  Hence $\pi_e$ induces a $\Z$-module isomorphism $(\pi_e)_*: \Z[\Lambda] \rightarrow \Z[F_e]$,
$\sum n_g g \mapsto \sum n_{g} \pi_e(g).$
\end{lemma}

\begin{proof} Let $Q = {\rm GAlex}(G_0,f) \cong {\mathscr H}(G_0,\{1\},f)$ where $1$ is the identity of the group $G_0$.   By  Lemma~\ref{HSV} in the Appendix the mapping
 $$ {\mathscr H}(G',{\rm Stab}_{G'}(e) , \phi) \rightarrow Q , \ \  {\rm Stab}_{G'}(e)g \mapsto eg$$
 where $\phi(g) = R_e^{-1}gR_e$ is an isomorphism of quandles.  By the same lemma since $G'$ has the smallest order  among all groups such that $Q \cong {\mathscr H}(G_1,H_1,f_1)$ it follows that $|Q| =|G_0| = |G'|$ and ${\rm Stab}_{G'}(e) = \{1\}$.

Since ${\rm Stab}_{G'}(e) = \{1\}$ it is clear that $\pi_e: g \mapsto eg$ is a bijection from $G'$ to $Q$. It remains to show that
$$\pi_e (C(x) \cap G') = {\rm inn}^{-1} (R_e).$$
For $g \in G'$ we have
 $g \in C(x) \cap G'$  $\iff$
  $gR_e = R_eg$$ \iff$
 $ugR_e = uR_eg$,  for all $u \in Q$ $\iff$
 $ug*e = (u*e)g = ug*eg$ for all $u \in Q$ $\iff$
 $v*e = v*eg$ for all $v \in Q$ $\iff$
$R_e = R_{eg}$ $\iff$
 $eg \in {\rm inn}^{-1} (R_e)$.
\end{proof}

\begin{theorem}\label{Psi=P}
If  $Q$ be a connected generalized Alexander quandle, $e \in Q$, $G = {\rm Inn}(Q)$ and $ x = R_e$
 then for a knot $K$
$$\Psi^e_Q(K) = (\pi_e)_*(P^x_G(K)),$$
where  $(\pi_e)_*$ is the $\Z$-module isomorphism defined in Lemma~\ref{lem:fiber}.  Thus $P_G^x$ and $\Psi^e_Q$ distinguish the same pairs of knots in this case.
\end{theorem}

\begin{proof} This is an immediate consequence of Lemma 3.12 of \cite{Eis2} using the inner augmentation ${\rm inn}: Q \rightarrow G = {\rm Inn}(Q)$  for $Q$ and the fact from  Lemma~\ref{lem:fiber} above that $\pi_e$ is a bijection from   $ \Lambda = C(x) \cap G'$ to $ F_e = {\rm inn}^{-1} (R_e)$.
Note that in Lemma 3.12 (\cite{Eis2} ) the notation $e^g$ is used in place of the notation $eg$ used here.
\end{proof}

\begin{remark} 
{\rm 
Note that the proof of Theorem~\ref{Psi=P}  fails for an arbitrary connected quandle $Q$. It depends on the fact that $\pi_e$ is a bijection. This will not be the case if in the action of ${\rm Inn}(G)'$ on $Q$ the stabilizer of an element $e \in Q$ is not equal to $\{1\}$. On the other hand, $(\pi_e)_*$ will be a $\Z$-module epimorphism in all cases. So it is possible that $P_G^x$ is a stronger invariant than $\Psi^e_Q$, but we have
no computational evidence for this.  This comparison depends on the inner augmentation. There may be other augmentations for other quandles.
}
\end{remark}

 \section{Non-faithful connected quandles }  \label{sec:nonfaith} 

 If we wish the invariant $\Psi^e_Q$ to be distinct from ${\rm Col}_Q$ we are constrainted to use only quandles that are not faithful. Leandro Vendramin in \cite{rig} has found all connected quandles of order $ \leq 47$. There are 790 such quandles.  We call these {\it Rig quandles}. In \cite{rig} the matrix of the $i$-th quandle of order $n$ is denoted by $SmallQuandle(n,i)$. We use the notation $Q(n,i)$ to denote the transpose of the matrix $SmallQuandle(n,i)$, thus changing Rig's left distributive quandles to the right distributive quandles used in knot theory.  Of the 790 Rig quandles only 66 are not faithful. It was shown in \cite{CSV}, Proposition 6.1,  that of the 66 all are abelian extensions except for the three quandles $Q(30,4), Q(36,58), $ and $Q(45,29)$.  Furthermore of the 66 non-faithful Rig quandles, 58 are abelian extensions by $\Z_2$.  

In our computations the 66 non-faithful Rig quandles did not suffice to distinguish $K$ from $s(K)$ when $K \neq s(K)$, $s \in \{r,m,rm\}$ for the prime knots with at most 12 crossings. We have in addition, 9 abelian extensions of orders greater than 47 that were constructed for use in \cite{CSV}. 
 To find more non-faithful connected quandles we constructed using GAP several thousand connected generalized Alexander quandles. As shown in Appendix B, it suffices to run through the small groups  $G$ in GAP,  for each group $G$ compute representatives $f$ of the conjugacy classes in ${\rm Aut}(G)$, and construct ${\rm GAlex}(G,f).$
When ${\rm GAlex}(G,f)$ is connected then as shown in Appendix B, it is an abelian extension of ${\rm inn}({\rm GAlex}(G,f))$  if and only if the subgroup ${\rm Fix}(G,f) = \{ x \in G \ | \ f(x) = x \}$  is abelian.
We give a proof of the following in Appendix B. Note that in this theorem ${\rm Fix}(G,f)$ need not be abelian.

\begin{theorem} \label{Extension} If $\Lambda$ is a subgroup of ${\rm Fix}(G,f)$ where $f \in {\rm Aut}(G)$ then  $${\rm GAlex}(G,f) \cong  \Lambda \times_\phi {\mathscr H}(G,\Lambda,f)$$ and the projection $\pi : \Lambda \times_\phi {\mathscr H}(G,\Lambda,f)  \rightarrow {\mathscr H}(G,\Lambda,f)$ given by $(\lambda,\Lambda g) \mapsto \Lambda g$ is equivalent to $p_\Lambda : {\rm GAlex}(G,f) \rightarrow  {\mathscr H}(G,\Lambda,f)$ defined by $p_\Lambda (g) = \Lambda g$.
 Moreover if $\Lambda = {\rm Fix}(G,f)$ then $\pi$ is equivalent to {\rm inn}.
\end{theorem}

This allows us to compute many cocycle invariants $\Phi_\phi(K) = \Psi_{ {\rm GAlex}(G,f)}^e(K)$ without finding explicit cocycles.

Note that in case  $\Lambda ={\rm Fix}(G,f)$ is not abelian ${\rm GAlex}(G,f)$ has the form $Q=\Lambda \times_\phi X$
for a {\it non-abelian cocycle}.  From this we see that in this case $\Psi^e_Q(K)$ may be considered
as an element of the group ring $\Z[\Lambda]$ of the non-abelian group $\Lambda$ and thus a generalization of the 2-cocycle invariant for abelian groups.

\section{Computational results}  \label{sec:compute} 

As mentioned in Section~\ref{sec:intro}, 
 in  \cite{CESY} it was shown that  there are 26 quandles that suffice to distinguish modulo the  group $\mathscr{G}  = \{1, m, r, rm \}$,  all pairs of the 2977 prime knots with crossing number at most 12 using the number ${\rm Col}_Q(K)$ of colorings of knot $K$. This holds also for the invariant $\Psi^e_Q$ since the coefficient of $e$ in $\Psi^e_Q(K)$ is ${\rm Col}_Q(K)/|Q|$. So to verify our conjecture that $\Psi^e_Q$ is a complete invariant for prime knots with at most 12 crossings it suffices to find for each such knot $K$ for which $K \neq s(K)$ for some $ s \in \mathscr{G}$ a quandle $Q$ such that $\Psi^e_Q(K) \neq \Psi^e_Q(s(K))$.

For knots $K$ and $K'$, we write $K = K'$ to denote that there is an orientation preserving homeomorphism of 
 $\Sym^3$ that takes $K$ to $K'$ preserving the orientations of $K$ and $K'$. We say that a knot has {\it symmetry} $s \in \{ m, r, rm \} $ if $K = s(K)$.  As in the definition of {\it symmetry type} in
\cite{KI} we say that  a knot $K$ is
\begin{enumerate}
\setlength{\itemsep}{-3pt}
\item { {\it reversible} if the only symmetry it has is $r$},
\item { {\it negative amphicheiral} if the only symmetry it has is $rm$},
\item {  {\it positive amphicheiral} if the only symmetry it has is $m$},
\item {{\it chiral} if it has none of these symmetries},
\item {{\it fully amphicheiral} if has all three symmetries $m, r, rm$}.
\end{enumerate}

 The symmetry type of each prime oriented knot on at most 12 crossings is given at \cite{KI}. Thus each of the $2977$ knots $K$ given there represents  $1$, $2$ or $4$ knots depending on the symmetry type. 
 Among the $2977$ knots, there are 
 $1580$ reversible, $47$ negative amphicheiral, $1$ positive amphicheiral,  
$1319$ chiral, and $30$ fully amphicheiral knots.  

For $s \in \mathscr{G}$ let $\mathscr{K}_s$ denote the set of prime oriented knots $K$ with at most 12 crossings 
such that $K \neq s(K)$. From the above we have
\begin{enumerate}
\setlength{\itemsep}{-3pt}
\item {$|\mathscr{K}_m|  = 1319 + 47 + 1580 = 2946$},
\item {$|\mathscr{K}_r|  = 1319 + 1 + 47 =  1367$},
\item {$|\mathscr{K}_{rm}|  = 1319 + 1 + 1580 = 2900$}.
\end{enumerate}
We have been able to find 60 connected quandles which will distinguish via the $\Psi$ invariant the following. 
\begin{enumerate}
\setlength{\itemsep}{-3pt}
\item {$K$ from $m(K)$ for all $K$ in $\mathscr{K}_m$},
\item {$K$ from $rm(K)$ for all $K$ in $\mathscr{K}_{rm}$ except  $12a_{0427}$,  the single positive amphicheiral prime knot with at most 12 crossings. Note that for this knot  $rm(K) = r(K)$}.
\item {$K$ from $r(K)$ for all  $K$ in $\mathscr{K}_r$, except the following 13 knots: 
$ 12a_{0059},12a_{0067},12a_{0292},12a_{0427},$ \break $12a_{0700},  12a_{0779}, 12a_{0926},12a_{0995},12a_{1012}, 12a_{1049},12a_{1055}, 12n_{0180},12n_{0761}$
all of which are chiral except for the positive amphicheiral knot $12a_{0427}$}.
\end{enumerate}

In all cases the quandles used are
either faithful and $\Psi$ reduces to ordinary coloring ${\rm Col}$ or else are abelian extensions whose projections are equivalent 
to the inner representation {\rm inn} and so the $\Psi$ invariant is the  2-cocycle invariant $\Phi$.

Quandle matrices for the 60 connected quandle used to distinguish knots $K$ from $s(K)$, $s \in \mathscr{G}$,  may be found
in the file  \url{http://shell.cas.usf.edu/~saito/QuandleColor/SixtyQuandles.txt}.   The list below gives some of their properties. The 26 quandles that distinguish the prime knots with at most 12 crossings modulo the group $\mathscr{G}$ may be found at \url {http://shell.cas.usf.edu/~saito/QuandleColor/}. Here we give a general description of the 60 quandles which were found mostly using GAP as discussed above. The orders of the 60 quandles  range from 18 to 1820.

35 of the 60 are of the form ${\rm GAlex}(G,f)$ where $G$ is a finite group.  Each of which is an abelian extension of ${\rm Image}({\rm inn})$ by an abelian group $\Lambda$. Below we omit the automorphism $f$.  These groups have order ranging from 24 to 660, the majority of which are finite simple non-abelian groups. 

23 of the 60 are conjugation quandles on groups ranging in order from 216 to 29120. We denote such a quandle by ${\rm Conj}(G)$ with the understanding that the underlying set of the quandle is a conjugacy class of $G$ and the quandle operation is conjugation. For conjugation quandles in the list, the largest group of order 29120 is the Suzuki group Sz(8) which provides a useful conjugation quandle of order 1820.  

There are two other quandles among the 60  that are neither conjugation quandles nor generalized Alexander quandles. One is 7 in the list:  the Rig quandle $Q(36,57)$, an abelian extension of $Q(12,10)$ by $\Z_3$ (see \cite{CSV}, proof of Theorem 7.1). The other is 22 in the list: an abelian extension of $Q(18,11)$ by $\Z_6$ which was also used in \cite{CSV}. 
\begin{multicols}{2}
\begin{enumerate}
\setlength{\itemsep}{-3pt}
\item{$18,{\rm Conj}(SmallGroup(216,90))$},
\item{$24,{\rm GAlex}(SL(2,3)),\Lambda = \Z_4$},
\item{$27,{\rm Conj}(SmallGroup(216,86))$},
\item{$27,{\rm Conj}(SmallGroup(486,41))$},
\item{$27,{\rm GAlex}(SmallGroup(27,3)),\Lambda = \Z_3$},
\item{$28,{\rm Conj}(SmallGroup(168,43))$},
\item{$36$, $Q(36,57) = \Z_3 \times_{\phi} Q(12,10)$},
\item{$40,{\rm Conj}(SmallGroup(320,1635))$},
\item{$60,{\rm GAlex}(A5),\Lambda = \Z_3$},
\item{$60,{\rm GAlex}(A5),\Lambda = \Z_3$},
\item{$60,{\rm GAlex}(A5),\Lambda = \Z_5$},
\item{$64,{\rm Conj}(SmallGroup(448,179))$},
\item{$64,{\rm Conj}(SmallGroup(768,1083508))$},
\item{$64,{\rm Conj}(SmallGroup(768,1083509))$},
\item{$72,{\rm GAlex}(Q8 : \Z_9),\Lambda = \Z_4$},
\item{$75,{\rm GAlex}(SmallGroup(75,2)),\Lambda = \Z_5$},
\item{$81,{\rm GAlex}(SmallGroup(81,10)),\Lambda = \Z_3$},
\item{$81,{\rm GAlex}(SmallGroup(81,12)),\Lambda = \Z_3$},
\item{$84,{\rm Conj}(SmallGroup(1512,779))$},
\item{$96,{\rm GAlex}(SmallGroup(96,203)), \\ \Lambda=\Z_4 \times \Z_2$}, %
\item{$108,{\rm GAlex}(SmallGroup(108,40)), \\ \Lambda = \Z_3$}, %
\item{$108$, $\Z_6 \times_{\phi} Q(18,11)$},
\item{$112,{\rm Conj}(SmallGroup(1344,816))$},
\item{$112,{\rm Conj}(SmallGroup(1344,816))$},
\item{$117,{\rm Conj}(SmallGroup(1053,51))$},
\item{$120,{\rm GAlex}(SL(2,5)),\Lambda = \Z_{10}$},
\item{$120,{\rm GAlex}(SL(2,5)),\Lambda = \Z_4$},
\item{$120,{\rm GAlex}(SL(2,5)),\Lambda = \Z_6$},
\item{$120,{\rm GAlex}(SL(2,5)),\Lambda = \Z_6$},
\item{$125,{\rm Conj}(SmallGroup(1500,36))$},
\item{$135,{\rm Conj}(SmallGroup(1620,421))$},
\item{$135,{\rm Conj}(SmallGroup(1620,421))$},
\item{$160,{\rm GAlex}(SmallGroup(160,199)), \\ \Lambda = \Z_4$}, %
\item{$162,{\rm Conj}(SmallGroup(1296,2890))$},
\item{$162,{\rm Conj}(SmallGroup(1296,2891))$},
\item{$168,{\rm Conj}(SmallGroup(1512,779))$},
\item{$168,{\rm GAlex}(PSL(3,2)),\Lambda = \Z_4$},
\item{$168,{\rm GAlex}(PSL(3,2)),\Lambda = \Z_4$},
\item{$168,{\rm GAlex}(PSL(3,2)),\Lambda = \Z_4$},
\item{$168,{\rm GAlex}(PSL(3,2)),\Lambda = \Z_3$},
\item{$168,{\rm GAlex}(PSL(3,2)),\Lambda = \Z_3$},
\item{$168,{\rm GAlex}(PSL(3,2)),\Lambda = \Z_7$},
\item{$192,{\rm Conj}(SmallGroup(1344,814))$},
\item{$243,{\rm GAlex}(SmallGroup(243,3)),\Lambda = \Z_3$},
\item{$336,{\rm GAlex}(SL(2,7)),\Lambda = \Z_8$},
\item{$351,{\rm Conj}((((\Z_3 \times \Z_3 \times \Z_3) : \Z_{13}) : \Z_3) : \Z_2)$},
\item{$360,{\rm GAlex}(A6),\Lambda = \Z_4$},
\item{$360,{\rm GAlex}(A6),\Lambda = \Z_5$},
\item{$504,{\rm GAlex}(PSL(2,8)),\Lambda = \Z_2$},
\item{$504,{\rm GAlex}(PSL(2,8)),\Lambda = \Z_3$},
\item{$504,{\rm GAlex}(PSL(2,8)),\Lambda = \Z_7$},
\item{$504,{\rm GAlex}(PSL(2,8)),\Lambda = \Z_9$},
\item{$504,{\rm GAlex}(PSL(2,8)),\Lambda = \Z_3$},
\item{$504,{\rm GAlex}(PSL(2,8)),\Lambda = \Z_3$},
\item{$660,{\rm GAlex}(PSL(2,11)),\Lambda = \Z_6$},
\item{$660,{\rm GAlex}(PSL(2,11)),\Lambda = \Z_6$},
\item{$660,{\rm GAlex}(PSL(2,11)),\Lambda = \Z_5$},
\item{$720,{\rm Conj}(M11)$},
\item{$990,{\rm Conj}(M11)$},
\item{$1820,{\rm Conj}(Sz(8))$}.
\end{enumerate}
\end{multicols}

In the file  \url{http://shell.cas.usf.edu/~saito/QuandleColor/PsiValues.txt}  we give a list of the quandles, knots, and the $\Psi$ values for all of the knots $K$ mentioned above for which the invariant distinguishes $K$ from $s(K)$ for some $s \in \{m,r,rm\}$.

In the file we indicate the 2977 non-trivial prime knots by ${\rm Knot}[i], i = 1, \dots, 2977$. The knot index $i$ is one less than the {\it name rank} of the knot given at \cite{KI}.  The quandle index $i$ represents the $i$\/th quandle in the file \url{SixtyQuandles.txt} mentioned above.

The file \url{PsiValues.txt} contains lists of the form
$$[[quandle \  index, knot \ index], \Psi^e_Q(K), \Psi^e_Q(m(K)),\Psi^e_Q(r(K)),\Psi^e_Q(rm(K))] . $$
For example the first list in the file is
$$[ [ 2, 1 ], [ 1, 0, 0, 4 ], [ 1, 0, 4, 0 ], [ 1, 0, 0, 4 ], [ 1, 0, 4, 0 ] . $$
This means that the 2nd quandle $Q_2$ in the file \url{SixtyQuandles.txt} distinguishes ${\rm Knot}[1] = 3_1$ from its mirror image
since $\Psi^e_{Q_2}(Knot[1]) =  [ 1, 0, 0, 4 ]$ and  $\Psi^e_{Q_2}(m(Knot[1])) =  [ 1, 0, 4, 0 ]$. Note that $Q_2$ is the Rig quandle $Q(24,3)$. 

Our {\it knot index}  $i$ is one less than the {\it name rank} of the knot given at \cite{KI}. Thus ${\rm Knot}[1]$ is the trefoil, whereas at \cite{KI} the trefoil has name rank 2. Braid representations of these knots (taken from \cite{KI}) can be found at
\url{http://shell.cas.usf.edu/~saito/QuandleColor/12965knotsGAP.txt}.

\section*{APPENDICES}
\appendix

Much of the development presented in these appendices is essentially due to Eisermann ( \cite{Eis1}, \cite{Eis2}, \cite{Eis3} ), 
however our emphasis is slightly different since we are particularly interested in constructing
large numbers of such quandles and determining when an extension ${\rm  inn } : Q \rightarrow {\rm inn}(Q)$  is or is not
equivalent to an abelian extension in order to test Conjecture~\ref{conjecture1}.
See also Remark~\ref{construction}.
We also need a simple criterion for isomorphism of generalized Alexander quandles, namely, Lemma~B.1 due to Hou.

\section{Abelian and Non-Abelian Extensions}

In this section we generalize the usual notion of abelian extension $\Lambda \times_\phi Q$ of a quandle $Q$ by an abelian group $\Lambda$  to an arbitrary group $\Lambda$. We note that such a generalization also appeared in  \cite{CEGS}. In the next section we show how to construct many such extensions (abelian and non-abelian) using generalize Alexander quandles, without finding the cocycles that determine the extensions. 

Let $Q$ be a quandle and $\Lambda$  be a group. Suppose there is a function $\phi: Q \times Q \rightarrow \Lambda$ such that the set $\Lambda \times Q$ becomes a quandle under the product $(\lambda,a)*(\mu,b) = (\lambda \phi(a,b),a*b)$.  Such a function $\phi$ will be called a cocycle even if $\Lambda$ is not abelian.
We denote this quandle by $\Lambda \times_{\phi} Q$.  The following lemma is straightforward.
 \begin{lemma} \label{cocycle} $\Lambda \times_{\phi} Q$ is a quandle if and only if the following conditions hold:

{\rm (i)} For all $a \in Q$,  $\phi(a,a) = 1.$

{\rm (ii)}  For all $a,b,c \in Q$, 
$\phi(a,b)\phi(a*b,c) = \phi(a,c)\phi(a*c,b*c).$
\end{lemma}

We recall from Eisermann \cite{Eis3} that an epimorphism $p: \tilde{Q} \rightarrow Q$ of quandles is called a {\it covering} if $p(y_1) = p(y_2)$ implies that $R_{y_1} = R_{y_2}$ and the {\it group of deck transformations} ${\rm Aut}(p)$ is the subgroup of ${\rm Aut}(\tilde{Q})$ consisting of those automorphisms $\lambda$ of  $\tilde{Q}$ satisfying $p\lambda= p$. Equivalently, $\lambda \in  {\rm Aut}(p)$ if $\lambda \in  {\rm Aut}(\tilde{Q})$ and for all $x \in Q$ and $y \in p^{-1}(x)$ we have $\lambda(y) \in p^{-1}(x)$, that is, if $F$ is a fiber of $p$  then $\lambda(F) \subset F$.

\begin{lemma} \label{coverings} Let  $p: \tilde{Q} \rightarrow Q$ be a covering.

{\rm (i)} For $\phi \in {\rm Aut}(p)$ and $\beta \in {\rm Inn}(\tilde{Q})$ we have 
$\phi \beta = \beta \phi.$

{\rm (ii)} If $\tilde{Q}$ is connected then ${\rm Aut}(p)$ acts freely on $\tilde{Q}$, that is, if $\phi \in {\rm Aut}(p)$ and for some $y_0 \in  \tilde{Q}$, $\phi(y_0) = y_0$ then $\phi$ acts as the identity on $\tilde{Q}$.

{\rm (iii)}  If $\tilde{Q}$ is connected,  $x \in Q$  and  $y_0 \in p^{-1}(x)$ the mapping $\tau: {\rm Aut}(p) \rightarrow p^{-1}(x)$ given by $\tau(\phi) = \phi(y_0)$ is a injection. Hence  $|{\rm Aut}(p)| \le | p^{-1}(x)|.$
\end{lemma}

\begin{proof} ( \cite{Eis3}  )
 {\rm (i)} Since $p$ is a covering and $p(\phi(y)) = p(y)$,  $z*\phi(y) = z*y$  for all $z$. Hence
$\phi R_y(x) = \phi(x*y) = \phi(x)*\phi(y) = \phi(x)*y = R_y \phi(x)$. So $\phi R_y = R_y \phi$ for all $y$. Since the $R_y$ generate ${\rm Inn}(\tilde{Q})$ {\rm (i)} holds.
{\rm (ii)} If $\tilde{Q}$ is connected,  for each $y \in \tilde{Q}$ there is a $\beta \in {\rm Inn}(\tilde{Q})$ such that $y = \beta(y_0)$. Thus 
       $$\phi(y) = \phi(\beta(y_0)) = \beta(\phi(y_0)) = \beta(y_0) = y.$$
{\rm (iii)} is immediate from {\rm (ii)}.
\end{proof}

We note that extensions $p:\tilde{Q} \rightarrow Q$ described in the following theorem  are called principal $\Lambda$-coverings  and are written $\Lambda  \curvearrowright \tilde{Q} \rightarrow Q$  in Eisermann \cite{Eis3}. The proof below uses the method of the proofs of Lemma 40 and Theorem 44 of \cite{Eis1}. One easily sees that the assumption that $\Lambda$ is abelian is not necessary.

\begin{theorem} \label{ExtCriterion} Let $p:\tilde{Q} \rightarrow Q$ be a covering. Let $\Lambda$ be a subgroup of 
the deck transformation group ${\rm Aut}(p)$ that acts freely and transitively on each fiber $p^{-1}(q)$, $q \in Q.$ Then $\tilde{Q}$ is isomorphic to $\Lambda \times_{\phi} Q$ for some cocycle $\phi$.
\end{theorem}

\begin{proof} We note that for $\lambda \in \Lambda$ and $a,b$ in $\tilde{Q}$  
since $p$ is a covering  we have $p(\lambda(b)) = p(b)$ and hence
$$a * \lambda(b) = a*b \quad {\rm  for \ all } \ a,b \in \Lambda.$$ 
Also we have by Lemma~\ref{coverings} (i) that $R_b(\lambda(a)) = \lambda(R_b(a))$,
hence
$$\lambda(a)*b = \lambda(a*b) \quad  {\rm  for\  all }\  a,b \in \Lambda.$$

Let $s:Q \rightarrow \tilde{Q}$ be a section of $p$, that is $p(s(q)) = q$ or $s(q) \in p^{-1}(q)$ for all $q \in Q$.
Then since $\Lambda$ acts freely and transitively on $p^{-1}(q)$ we have that $\lambda \mapsto \lambda s(q)$ is a bijection from $\Lambda$ to the fiber $p^{-1}(q).$ In particular we have 
$$p^{-1}(q) = \Lambda s(q) \quad {\rm  for \ all } \ q \in Q.$$
Now since $p(s(x)*s(x)) = p(s(x))*p(s(y)) = x*y$ for $x,y \in \tilde{Q}$ there exists unique  $\phi(x,y) \in \Lambda$ such that
$$s(x)*s(y) = \phi(x,y)s(x*y) \quad {\rm  for \ all } \ x, y \in Q.$$

Now we claim that $\phi: Q \times Q \rightarrow \Lambda$ is a cocycle and that the mapping
 $f: \Lambda \times_{\phi} Q \rightarrow \tilde{Q}$ defined by $f(\lambda, a) = \lambda ( s(a))$
is an isomorphism of quandles. $f$ is clearly a bijection since each element of $\tilde{Q}$ can be uniquely represented in the form $\lambda s(a)$ for $\lambda \in \Lambda$ and $a \in Q.$
So it suffices to show that $f$ is a homomorphism. We compute:
\begin{eqnarray*}
& & f((\lambda,a)*(\mu,b)) = 
f(\lambda \phi(a,b), a*b) = 
\lambda( \phi(a,b) (s(a*b)) \\ & & 
 = 
\lambda (s(a)*s(b))
=
\lambda(s(a))*s(b) =
\lambda(s(a))*\mu(s(b)) =
f((\lambda,a))*f((\mu,b)).
\end{eqnarray*}
Note that this shows that $\Lambda \times_{\phi} Q$ is isomorphic as a magma to the quandle $\tilde{Q}$ and so is indeed a quandle and hence $\phi$ does satisfy the cocyle condition in Lemma~\ref{cocycle}.
\end{proof}

\section{Generalized Alexander Quandles}\label{sec:galex}

We remind the reader that all quandles and groups in this paper are finite.

In this section we show how we justify our construction of most of the quandles used in this paper.
This method allows us to easily construct using GAP many  connected abelian
and non-abelian extensions $\Lambda \times_\phi Q$ without finding the cocycles $\phi$.

Given a group $G$ and an automorphism $f$ of $G$ let ${\rm GAlex}(G,f)$ be the quandle
with underlying set  $G$ and product $a*b = f(ab^{-1})b$ for $a,b \in G$.
It is well-known that this is a quandle. We call this a generalized Alexander quandle. 

In the case where $G$ is abelian ${\rm GAlex}(G,f)$ is called an Alexander quandle and has been much studied. Note that in general ${\rm GAlex}(G,f)$ need not be connected.  We note that in \cite{Eis3} Eisermann writes ${\rm Alex}(G,f)$ for our ${\rm GAlex}(G,f)$ and, breaking with tradition, calls this an Alexander quandle.  The following is useful for constructing generalized Alexander quandles up to isomorphism.

\begin{lemma} \label{Hou} (Hou \cite{Hou}) If  ${\rm GAlex}(G,f)$ and ${\rm GAlex}(G,g)$ are connected then they are isomorphic if and only if
$f$ is conjugate to $g$ in ${\rm Aut}(G)$. Also for connected quandles if  ${\rm GAlex}(G_1,f_1) \cong {\rm GAlex}(G_2,f_2)$ (as quandles) then $G_1 \cong G_2$ (as groups). 
\end{lemma}

For a group $G$ and $f \in {\rm Aut}(G)$ we let ${\rm Fix}(G,f) = \{x \in G: f(x) = x \}$.
Using the notation of \cite{HSV} if  $H$ is a subgroup of ${\rm Fix}(G,f)$, then ${\mathscr H}(G,H,f) $ is  the quandle  with underlying set $\{Hg: g \in G\}$ with product
given by $Ha*Hb = Hf(ab^{-1})b$. It is well-known that this is a quandle and that every
connected quandle can be so represented. In the terminology of \cite{HSV} a quandle that is isomorphic to some ${\mathscr H}(G,H,f)$ is said to admit a  homogeneous  representation.  Clearly ${\mathscr H}(G,\{1\},f) \cong {\rm GAlex}(G,f)$. The following result
is useful for determining when a connected quandle is a generalized Alexander quandle.  We  denote the derived subgroup of $G$ by $G'$. 

\begin{lemma} (\cite{HSV} ) \label{HSV}If $Q$ is a finite connected quandle, $G = {\rm Inn}(Q)'$,
$e \in Q$,  $\phi(g) =R_e^{-1}gR_e$, for $g \in G$, then the mapping
$$\pi_e:{\mathscr H}(G,{\rm Stab}_G(e), \phi) \rightarrow Q$$
given by $\pi_e({\rm Stab}_G(e)g) = eg$ is a quandle isomorphism.  Moreover if $G_1$ is of smallest order among all groups such that $Q \cong {\mathscr H}(G_1,H_1,f_1)$, then $G_1 \cong {\rm Inn}(Q)'$.
\end{lemma}
\begin{corollary} \label{minG}A connected quandle $Q$ is a generalized Alexander quandle ${\rm GAlex}(G,f)$ if and only if  one of the following holds:
 (i) $|Q| = |{\rm Inn}(Q)'|$, or
 (ii) ${\rm Stab}_{{\rm Inn}(Q)'}(e) = \{1\}$, for any $e \in Q$.
(If either (i) or (ii) holds then both hold.)  
Furthermore in this case we may take $G = {\rm Inn}(Q)'$ and for any $e \in Q$,  $f(g) = R_e^{-1}gR_e$, for $g \in G$.
\end{corollary} 

Note that in the above corollary generally $R_e$ will not be in $G = {\rm Inn}(Q)'$ but $f$ is well-defined on $G$
since ${\rm Inn}(Q)' $ is a normal subgroup of ${\rm Inn}(Q)$.

\begin{corollary} Let  $p: \tilde{Q} \rightarrow Q$ be a covering such that
${\rm Aut}(p)$ acts transitively on each fiber.
Then if $y \in p^{-1}(x)$ for some $x \in Q$ and for some $\beta \in {\rm Inn}(Q)$  we have $\beta(y) = y$
then $beta$ acts as the  identity on the fiber.  Hence if $y_0, y_1, y_2 \in p^{-1}(x)$   and  $(y_0,y_1)$ and $(y_0, y_2)$ are in the same orbit of 
${\rm Inn}(\tilde{Q})$ acting on pairs, then $y_1 = y_2.$
\end{corollary}

\begin{proof} Since ${\rm Aut}(p)$ acts transitively on each fiber,  for each $y' \in p^{-1}(x)$ there exist $\gamma \in {\rm Aut}(p)$ such that $\gamma(y) = y'$. Then since by Lemma~\ref{coverings} (i) $\beta$ and $\gamma$ commute we have $\beta(\gamma(y)) = \gamma(\beta(y)) = \gamma(y)$ and hence $\beta(y') = y'$ for all  $y' \in p^{-1}(x)$. 
\end{proof}

\begin{lemma} Let  $\tilde{Q}$ be a connected quandle and $p: \tilde{Q} \rightarrow Q$ be a covering.
If  $\Lambda$ is a subgroup of ${\rm Aut}(p)$ which acts transitively on each fiber
then $\Lambda = {\rm Aut}(p)$ and each fiber has cardinality $|\Lambda|$.
\end{lemma}

\begin{proof} By Lemma \ref{coverings} ${\rm Aut}(p)$ acts freely on each fiber $F$. By hypothesis $\Lambda$ and hence ${\rm Aut}(p)$ acts transitively on each fiber $F$, so if $y_0$ is any element of  $F$ then the mapping $\phi \mapsto \phi(y_0)$ is a bijection from ${\rm Aut}(p)$ to $F$.   If $\Lambda \neq {\rm Aut}(p)$ then we would have
for any fiber $|F| = |\Lambda| < |{\rm Aut}(p)| = |F|,$ a contradiction.
\end{proof}

\begin{lemma} \label{Inn} (Cf. Example 4.16 in \cite{Eis3})\label{extension} If ${\rm GAlex}(G,f)$ is a  generalized Alexander quandle 
and $\Lambda$ is a subgroup of ${\rm Fix}(G,f)$ then the mapping $p_\Lambda:{\rm GAlex}(G,f) \rightarrow {\mathscr H}(G,\Lambda,f)$ 
defined by $p_\Lambda(g) = \Lambda g$ is a covering,  $\Lambda = {\rm Aut}(p_\Lambda)$ and $\Lambda$ acts freely and transitively on each fiber of $p_\Lambda$.   Moreover $p_\Lambda$ is equivalent to the right translation mapping ${\rm inn} : {\rm GAlex}(G,f) \rightarrow {\rm inn}({\rm GAlex}(G,f))$  if and only if $\Lambda = {\rm Fix}(G,f)$.
\end{lemma}

\begin{proof} If $p_\Lambda(a) = p_\Lambda (b)$ then $\Lambda a = \Lambda b$. It follows that $a b^{-1} \in \Lambda \subset {\rm Fix}(G,f)$. Hence $f(ab^{-1}) = ab^{-1}$, implying that 
$ f(b^{-1})b = f(a^{-1})a$
 and so for all $x \in {\rm GAlex}(G,f)$, 
 $$x*a = f(x)f(a^{-1})a = f(x)f(b^{-1})b = x*b. $$
 Thus $p_\Lambda$ is a covering.  Clearly $p_\Lambda^{-1}(\Lambda g) = \Lambda g$,  the right coset of $\Lambda$ in $G$ containing $g$.  Define an action of $\Lambda$ on ${\rm GAlex}(G,f)$ by setting for $\lambda \in \Lambda$, $\lambda g $ to be the product in $G$.  Now  using the fact that $f(\lambda) = \lambda$, we have 
 $$\lambda(a*b) = \lambda f(ab^{-1})b =
f(\lambda)f(a)f(b^{-1})f(\lambda^{-1})\lambda b = f(\lambda a (\lambda b)^{-1})\lambda b =
(\lambda a)*(\lambda b). $$
 Thus this action of $\lambda$ is an automorphism of ${\rm GAlex}(G,f)$.  Since $ \Lambda $ is a subgroup, clearly it acts freely and transitively on any fiber $\Lambda g$ of $p_\Lambda$.  Now for $\Lambda = {\rm Fix}(G,f)$ define
$\tau: {\rm  inn} ({\rm GAlex}(G,f)) \rightarrow {\mathscr H}(G,\Lambda,f)$ by $\tau(R_g) = \Lambda g$. One easily checks that $\tau$ is  well-defined and bijective. Note that the product in ${\rm Inn}({\rm GAlex}(G,f))$ is given by conjugation and $R_a*R_b = R_{a*b}$ so it suffice to observe that 
$$\tau( R_a*R_b) = R_{a*b}= \Lambda f(ab^{-1})b = \Lambda a*\Lambda b.$$
  Note that $|{\rm Image}({\rm inn}) | = |G|/|{\rm Fix}(G,f)|$ but for any proper subgroup $\Lambda$ of ${\rm Fix}(G,f)$,   ${\rm Image}(p_\Lambda) = |G|/|\Lambda| >  {\rm  |Image}({\rm inn}) |.$
This shows that only for $\Lambda = {\rm Fix}(G,f)$ is $p_\Lambda$ equivalent to ${\rm inn}$.
\end{proof}

\begin{theorem}[Theorem~\ref{Extension}] 
 If $\Lambda$ is a subgroup of ${\rm Fix}(G,f)$ where $f \in {\rm Aut}(G)$ then  $${\rm GAlex}(G,f) \cong  \Lambda \times_\phi {\mathscr H}(G,\Lambda,f)$$ and the projection $\pi : \Lambda \times_\phi {\mathscr H}(G,\Lambda,f)  \rightarrow {\mathscr H}(G,\Lambda,f)$ given by $(\lambda,\Lambda g) \mapsto \Lambda g$ is equivalent to $p_\Lambda$ as defined in Lemma~\ref{extension}. Moreover if $\Lambda = {\rm Fix}(G,f)$ then $\pi$ is equivalent to {\rm inn}.
\end{theorem}

\begin{proof} In the notation of Theorem~\ref{ExtCriterion}, the section 
$s: {\mathscr H}(G,\Lambda,f) \rightarrow {\rm GAlex}(G,f)$
 is given by  letting $s(\Lambda g)$ be any representative
of the coset $\Lambda g.$ Hence we have $\Lambda g = \Lambda s(\Lambda g)$. 
Then the isomorphism 
$f:\Lambda \times_\phi {\mathscr H}(G,\Lambda,f) \rightarrow {\rm GAlex}(G,f)$ 
is given by $f((\lambda,\Lambda g) )= \lambda s(\Lambda g)$.
 Then
$$p_\Lambda(f((\lambda,\Lambda g))) =
 p_\Lambda (\lambda s(\Lambda g))) =
\Lambda \lambda s(\Lambda g) =
\Lambda g =
\pi ((\lambda,\Lambda g)).$$
Hence $p_\Lambda f = \pi$ and so $p_\Lambda$ is equivalent to $\pi.$  The last statement
follows from Lemma~\ref{Inn}.
\end{proof}

\begin{remark} 
{\rm
From Lemma~\ref{extension} if ${\rm GAlex}(G,f)$ is not faithful, then ${\rm Fix}(G,f)$ must be non-trivial and since any non-trivial group contains a non-trivial abelian subgroup, by  Theorem~\ref{Extension} the quandle ${\rm GAlex}(G,f)$ is always an abelian extension of some quandle. However, ${\rm inn}: {\rm GAlex}(G,f) \rightarrow {\rm inn}({\rm GAlex}(G,f))$ need not be an abelian extension as we see using Example~\ref{example} and the following lemma.
}
\end{remark}

\begin{lemma}\label{GAlexAbelian} Let ${\rm GAlex}(G,f)$ be connected. If ${\rm GAlex}(G,f) \cong \Lambda \times_{\phi} Q$  where $|\Lambda|=|{\rm Fix}(G,f)|$ then $\Lambda \cong {\rm Fix}(G,f)$ as groups.  Hence if ${\rm Fix}(G,f)$ is not abelian then
 $\Lambda \times_{\phi} Q$ is not an abelian extension of $Q.$
\end{lemma}

\begin{proof} Let $\varphi: {\rm GAlex}(G,f) \rightarrow \Lambda \times_{\phi} Q$  be an isomorphism
and let $\pi: \Lambda \times_{\phi} Q \rightarrow Q$
be the projection from $\Lambda \times_{\phi} Q$ to $Q$. Then $p = \pi \varphi : {\rm GAlex}(G,f) \rightarrow Q$ is a covering
which by definition is isomorphic to the covering $\pi$ in the category of coverings of $Q$ (\cite{Eis3} ).
It follows that the groups ${\rm Aut}(p)$ and ${\rm Aut}(\pi)$ are isomorphic and the cardinality of the fibers of $p$ and $\pi$ are equal.  Clearly ${\rm Aut}(\pi) \cong \Lambda$. Now for the covering $p:{\rm GAlex}(G,f) \rightarrow Q$ we have for $g \in {\rm GAlex}(G,f)$,  $h \in p^{-1}(p(g))$, $p(h) = p(g)$ and hence $R_h = R_g$. This implies that $f(xg^{-1})g = f(xh^{-1})h$ for all $x \in G$ and hence
$f(g^{-1})g = f(h^{-1})h$ and $f(gh^{-1}) = gh^{-1}$, thus $h \in {\rm Fix}(G,f)g$. So $p^{-1}(p(g)) \subseteq {\rm Fix}(G,f)g$. Since  $|p^{-1}(p(g))| = |{\rm Fix}(G,f)g|$ we must have $p^{-1}(p(g)) = {\rm Fix}(G,f)g$ and it follows that 
${\rm Aut}(p) = {\rm Fix}(G,f)$. This proves that $\Lambda \cong {\rm Fix}(G,f)$.  
\end{proof}

\begin{remark} \label{EisQuandles} 
{\rm
The quandles ${\rm GAlex}(G,f)$ have an alternative description as described by Eisermann \cite{Eis1}. Also compare Corollary~\ref{minG} above.
Namely, let $\tilde{G}$ be a group with $x \in \tilde{G}$. If $G$ is the derived subgroup of $\tilde{G}$ then $f(g) = x^{-1}gx$, $g \in G$, defines an automorphism of $G$. Then by  Lemma 25 of \cite{Eis1} together with Remark 27 of the same paper, ${\rm GAlex}(G,f)$ is connected if $\tilde{G}$ is generated as a group by the conjugacy class
$x^{\tilde{G}}$  of $x$ in $\tilde{G}$. There it is proved that $x^G = x^{\tilde{G}}$ and it is easy to check that the mapping  $g: {\rm GAlex}(G,f) \rightarrow x^G$ defined by $g(a)=  a^{-1}x a$ for $a \in G$ is a covering equivalent to ${\rm inn}.$   Note that the product in $x^G$ is conjugation.
}
\end{remark}

\begin{remark} \label{construction}
{\rm 
The method of constructing up to isomorphism all quandles of the form ${\rm GAlex}(G,f)$ via the method in Remark~\ref{EisQuandles} is inefficient since in general there are many groups $\tilde{G}$ whose derived subgroup is a given group $G$. Thus using such a method makes isomorphism testing difficult. On the other hand using Lemma~\ref{Hou} and GAP it is easy to find for each small group $G$ representatives $f$ of the conjugacy classes of ${\rm Aut}(G)$ to construct distinct quandles of the form ${\rm GAlex}(G,f)$. One only needs to check connectivity which is easily done directly and to compute the group ${\rm Fix}(G,f)$ to determine whether ${\rm GAlex}(G,f)$ is an abelian extension.
}
\end{remark}

\begin{example} \label{example} 
{\rm
Using Remark~\ref{EisQuandles} we give a simple class of connected extensions   $ \Lambda \times_\phi Q$ where the projection $\pi:  \Lambda \times_\phi Q \rightarrow Q$ is equivalent to ${\rm inn}$
and $\Lambda$ is a non-abelian group. Let $n \geq 5$ and let $x=(1 \ 2) \in S_n$.  Let  $f: A_n \rightarrow A_n$ be defined by  
$f(g) = x g x$, $g \in A_n$. $f$ is an automorphism of $A_n$ since $A_n$ is a normal subgroup of $S_n$.
By Theorem \ref{Extension} it suffices to note that $\Lambda = {\rm Fix}(A_n,f)$ is not abelian.
For example let $\alpha = (1  2) (3 \ 4 )$ and $\beta = (1 \ 2)(3 \ 5)$ clearly $\alpha, \beta \in {\rm Fix}(A_n,f)$
but $\alpha \beta(3) = 4$ and $\beta \alpha(3) = 5 $, so $\alpha$ and $\beta$ do not commute. In fact, in this case, $\Lambda \cong S_{n-2}$.
We also note that the image of the covering $p_\Lambda: {\rm GAlex}(A_n,f) \rightarrow {\mathscr H}(A_n,\Lambda,f)$ as defined in Lemma~\ref{extension}.
is isomorphic to the conjugation quandle on the conjugacy class consisting of all transpositions in $S_n$.
 Since $S_n$ is generated by all  transpositions,  by Remark~\ref{EisQuandles} the quandles ${\rm GAlex}(A_n,f)$ are connected.
 
}
\end{example}

\subsection*{Acknowledgements}
MS was partially supported by
NIH R01GM109459.

\end{document}